\documentclass{article}
\usepackage{amssymb,amsmath,amsthm,graphicx}

\def\tr{\triangleright}

\newtheorem{theorem}{Theorem}

\newtheorem{corollary}[theorem]{Corollary}

\theoremstyle{definition}
\newtheorem{example}{Example}
\newtheorem{definition}{Definition}
\newtheorem{remark}{Remark}

\date{}

\title{\Large \textbf{Quandle Coloring Quivers}}

\author{Karina Cho\footnote{Email: kcho@hmc.edu. Supported by the Giovanni Borrelli Mathematics Fellowship}
\and
Sam Nelson\footnote{Email: Sam.Nelson@cmc.edu. Partially supported by Simons Foundation collaboration grant 316709}}

\begin{document}
\maketitle

\begin{abstract}
We consider a quiver structure on the set of quandle colorings of an oriented
knot or link diagram. This structure contains a wealth of knot and link
invariants and provides a categorification of the quandle counting invariant
in the most literal sense, i.e., giving the set of quandle colorings the
structure of a small category which is unchanged by Reidemeister moves. We
derive some new enhancements of the counting invariant from this quiver
structure and show that the enhancements are proper with explicit examples.
\end{abstract}

\parbox{5.25in}{\textsc{Keywords:} Quandles, Enhancements, Quivers,
Categorified link invariants, In-degree quiver polynomials

\smallskip

\textsc{2010 MSC:} 57M27, 57M25}

\section{\large\textbf{Introduction}}\label{I}

\textit{Quandles} were defined in \cite{J} (see also \cite{M} where
they are called \textit{distributive groupoids}) as an abstract algebraic
structure generalizing the Wirtinger presentation of the knot group, with
axioms encoding the Reidemeister moves. Quandles have
been used to define many useful invariants of oriented links and knots.
In particular, associated to every finite quandle $X$ is the positive
integer-valued \textit{quandle counting invariant} $\Phi_X^{\mathbb{Z}}$
which can be easily computed from a diagram of an oriented knot
or link $L$.

Starting in \cite{CJKLS} and continuing in many recent works (see
\cite{CEGS,EN,GN} etc.), oriented link invariants known as \textit{enhancements}
of the quandle counting invariant have been defined and studied. An enhancement
is a generally stonger invariant which determines the counting invariant. The
original example is the \textit{2-cocycle invariant} where a quandle 2-cocycle
is used to define an invariant of quandle-colored oriented knots and links
known as a \textit{Boltzmann weight}. The multiset of these Boltzmann weights
is then a stronger invariant which has the counting
invariant as its cardinality. For ease of comparison,
the multiset is often written as a polynomial by making the Boltzmann weights
exponents of a formal variable $u$ with multiplicities as coefficients; in this
format we can recover the counting invariant by evaluating the polynomial
at $u=1.$ See \cite{CJKLS,EN} for more.

Related to enhancement is the notion of \textit{categorification}, whose many
recent examples include Khovanov homology in various flavors
(see e.g. \cite{K,OS}) as well as Knot Floer homology (e.g. \cite{M1}) and more.
The general idea is to replace basic structures with richer algebraic
structures to get finer and more sensitive information,  e.g. replacing
elements of a set with objects in a category and operations with morphisms.
In particular, categorified knot and link invariants are examples of 
enhancements of knot and link invariants.

In this paper we introduce the \textit{quandle coloring quiver}, a directed
graph-valued enhancement
of the quandle counting invariant, and derive from it several
simpler enhancements. This enhancement categorifies the quandle counting
invariant in the sense that it endows the set of quandle colorings of a knot
with the structure of a small category in a way which is invariant under
Reidemeister moves.

The paper is organized as follows. In Section \ref{QB} we recall the basics
of quandles and the quandle counting invariant. In Section \ref{QCQ} we define
the quandle coloring quiver and derive from it several enhancements of the
quandle counting invariant. We provide examples which demonstrate that the new
enhancements are proper. We conclude in Section \ref{Q} with some questions for
future work.

\section{\large\textbf{Quandle Basics}}\label{QB}

We begin with some relevant definitions and examples.

\begin{definition}
A set $X$ equipped with a binary operation $\tr$ is a \textit{quandle} if it satisfies
\begin{enumerate}
    \item $x \tr x = x$ for all $x \in X$,
\item for each $y \in X$, the map $f_y: Q \to Q$ defined by $f_y(x) = x \tr y$ is a bijection, and
    \item $(x\vartriangleright y)\vartriangleright z = (x\vartriangleright z)\vartriangleright(y \vartriangleright z)$ for all $x,y,z \in X$.
\end{enumerate}
\end{definition}

\begin{example}\label{ex:conjugation}
Let $G$ be a group where the $\tr$ operation is $n$-fold conjugation: $x \tr y = y^n x y^{-n}$. Then $G$ is a quandle.
\end{example}

\begin{example}\label{ex:dihedral}
Let $X = \mathbb{Z}/n\mathbb{Z}$ where the $\tr$ operation is defined by $x \tr y \equiv 2y - x$ (mod $n$). Then $X$ is a quandle called the \textit{dihedral quandle}.
\end{example}

Finite quandles, which are a focus of this paper, can be fully represented by operation tables. The operation table for the dihedral quandle with $n = 3$ is shown below.

\[\begin{array}{r|rrr}
\tr & 0 & 1 & 2 \\ \hline
  0 & 0 & 2 & 1 \\
  1 & 2 & 1 & 0 \\
  2 & 1 & 0 & 2
\end{array}\]

Next, we give a combinatorial definition of the \textit{fundamental quandle} of a knot or link. Let $L$ be an oriented knot or link with $n$ arcs in its realization. We assign a label to each arc. At each crossing, we assign the relation given by the picture below.
\[\scalebox{0.2}{\includegraphics{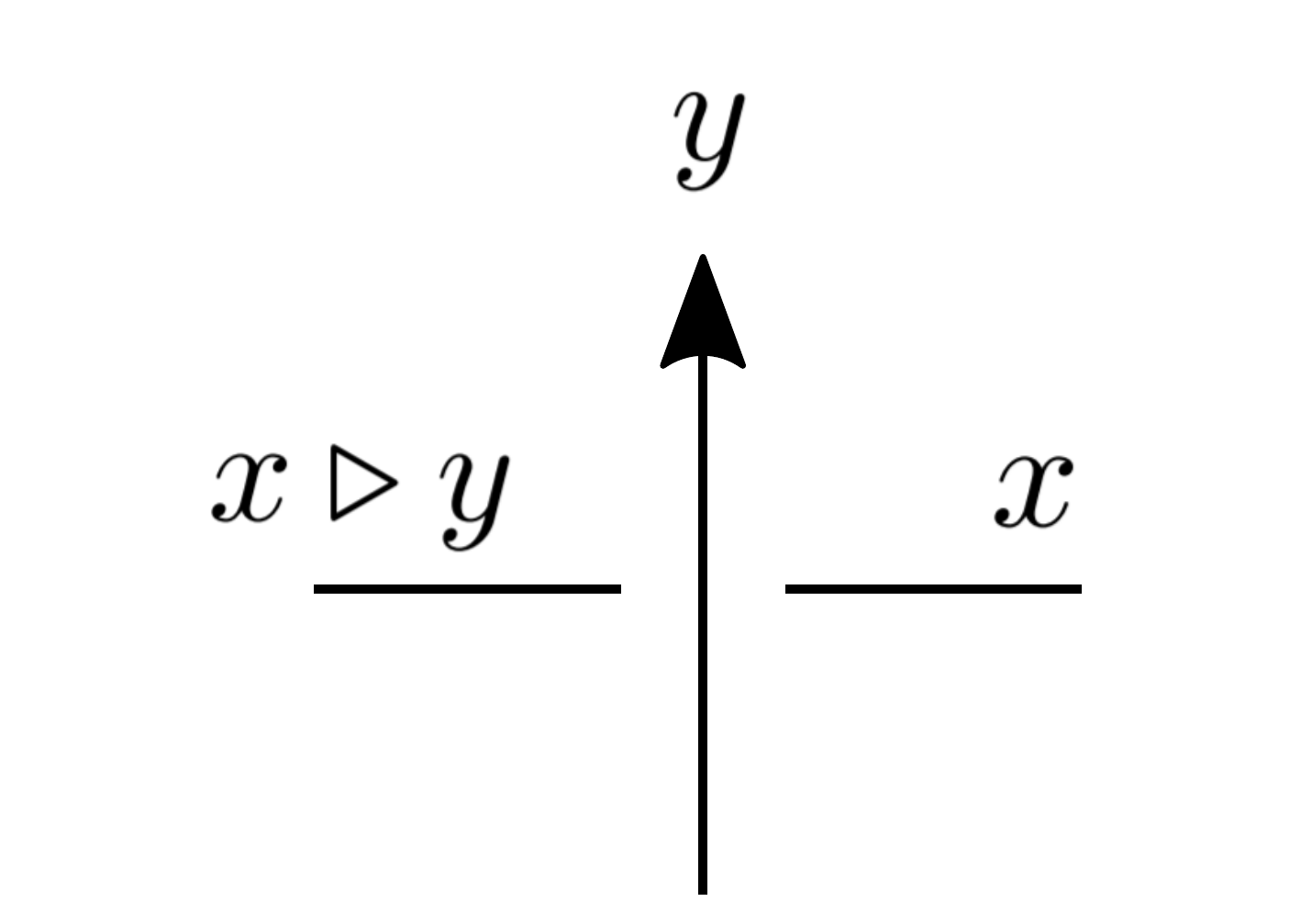}}\]

Note that only the orientation of the overarc matters. The fundamental quandle $Q(L)$ is the set of equivalence classes of quandle words generated by the arc labels under the equivalence relations given by the crossings and the quandle axioms.

\begin{example}\label{ex:fundquandle}
We will compute $Q(3_1)$ for the trefoil knot. Here the arcs are labelled $a,b,c$. \[\scalebox{0.28}{\includegraphics{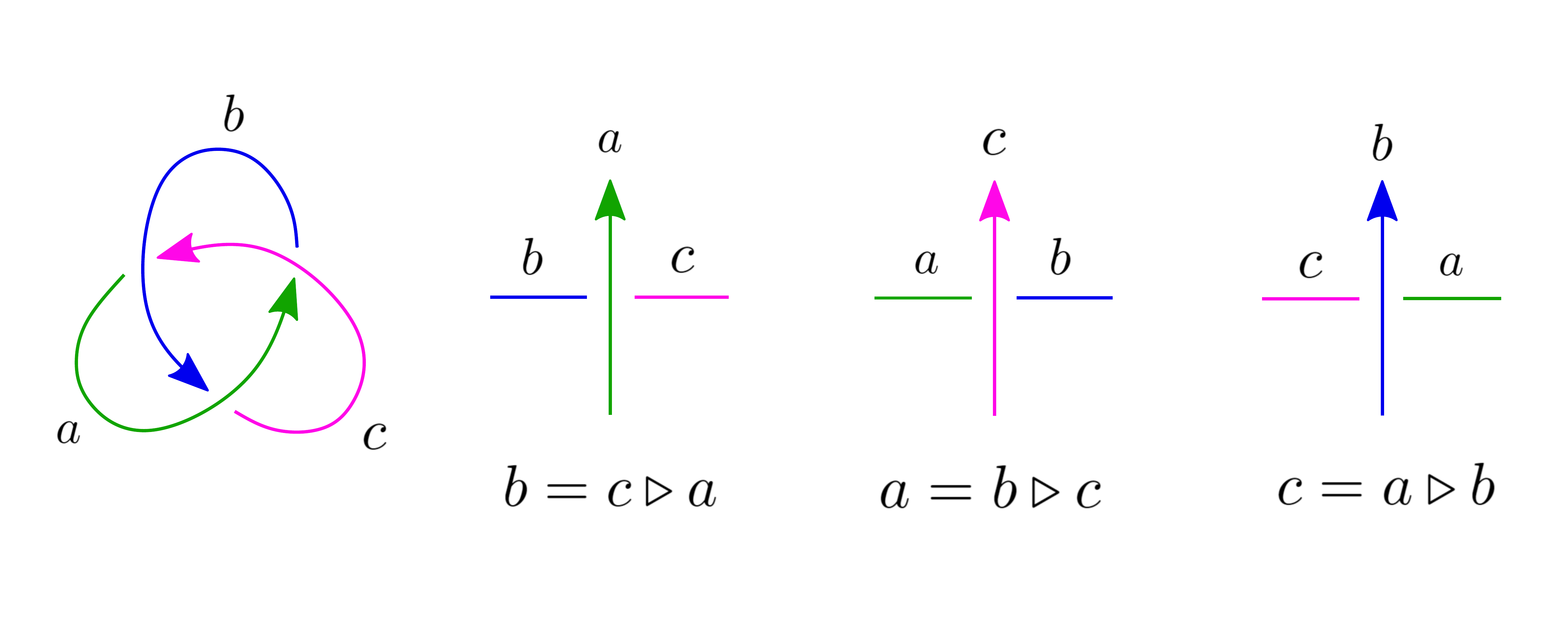}}\]

Then the fundamental quandle has presentation $Q(3_1) = \langle a,b,c \, |  \, b = c \tr a, a = b \tr c, c = a \tr b \rangle$.
In this case, the fundamental quandle is infinite. Note that by construction, any knot or link with finitely many arcs in its diagram will have a finitely generated fundamental quandle.

\end{example}

\begin{definition}
Let $X, Y$ be quandles with multiplication operations indicated by  $\vartriangleright_X$ and $\vartriangleright_Y$ respectively. A map $f:X \to Y$ is a \textit{quandle homomorphism} given that $f(a \vartriangleright_X b) = f(a)\vartriangleright_Y f(b)$ for any $a,b \in X$.
\end{definition}

\begin{definition}
Let $L$ be an oriented knot or link and $X$ a finite quandle called the coloring quandle. The coloring space $\mathrm{Hom}(Q(L),X)$ is the space of quandle homomorphisms from $Q(L) \to X$. The \textit{quandle counting invariant} is the cardinality of the coloring space, $|\mathrm{Hom}(Q(L),X)|$, which we will denote by $\Phi_X^{\mathbb{Z}}(L)$.
\end{definition}

\begin{remark}
Combinatorially, each element $\phi \in \mathrm{Hom}(Q(L),X)$ can be represented as a ``coloring'' of the diagram of $L$ by colors from $X$. Using this analogy, a valid coloring is an assignment of an element from $X$ to each arc in $L$'s link diagram in a way that respects the quandle operation of $X$ at each crossing. This makes sense since each arc is a generator for $Q(L)$ and a homomorphism is uniquely defined by the mapping of generators. Since there are finitely many ways to assign colors from a finite quandle $X$ to a link diagram with $n$ arcs, the quandle counting invariant always yields an integer value. From here onwards we will think of homomorphisms in $\mathrm{Hom}(Q(L),X)$ as colorings of link diagrams.
\end{remark}

\begin{example}\label{ex:coloringsoftrefoil}
We will compute the quandle counting invariant $|\mathrm{Hom}(Q(L),X)|$ where $L$ is the trefoil as orientied in example \ref{ex:fundquandle} and $X$ is the dihedral quandle on 3 elements. Recall that $Q(L)$ has presentation $Q(L) = \langle a,b,c \, |  \, b = c \tr a, a = b \tr c, c = a \tr b \rangle$ and $X$ has the multiplication table shown below.

\[\begin{array}{r|rrr}
\tr & 0 & 1 & 2 \\ \hline
  0 & 0 & 2 & 1 \\
  1 & 2 & 1 & 0 \\
  2 & 1 & 0 & 2
\end{array}\]

To count the homomorphisms, we need to count the different ways we can validly map the generators $a,b,c$ to elements of $X$. As seen in example \ref{ex:fundquandle}, we must satisfy $c = a \tr b$, so once colors for $a$ and $b$ are chosen, the color for $c$ is determined by the multiplication table. There are 3 choices each for $a$ and $b$, so we see that $|\mathrm{Hom}(Q(L),X)| = 3 \cdot 3 = 9$. These 9 colorings are depicted below where 0 is green, 1 is blue, and 2 is pink.

\[\scalebox{0.3}{\includegraphics{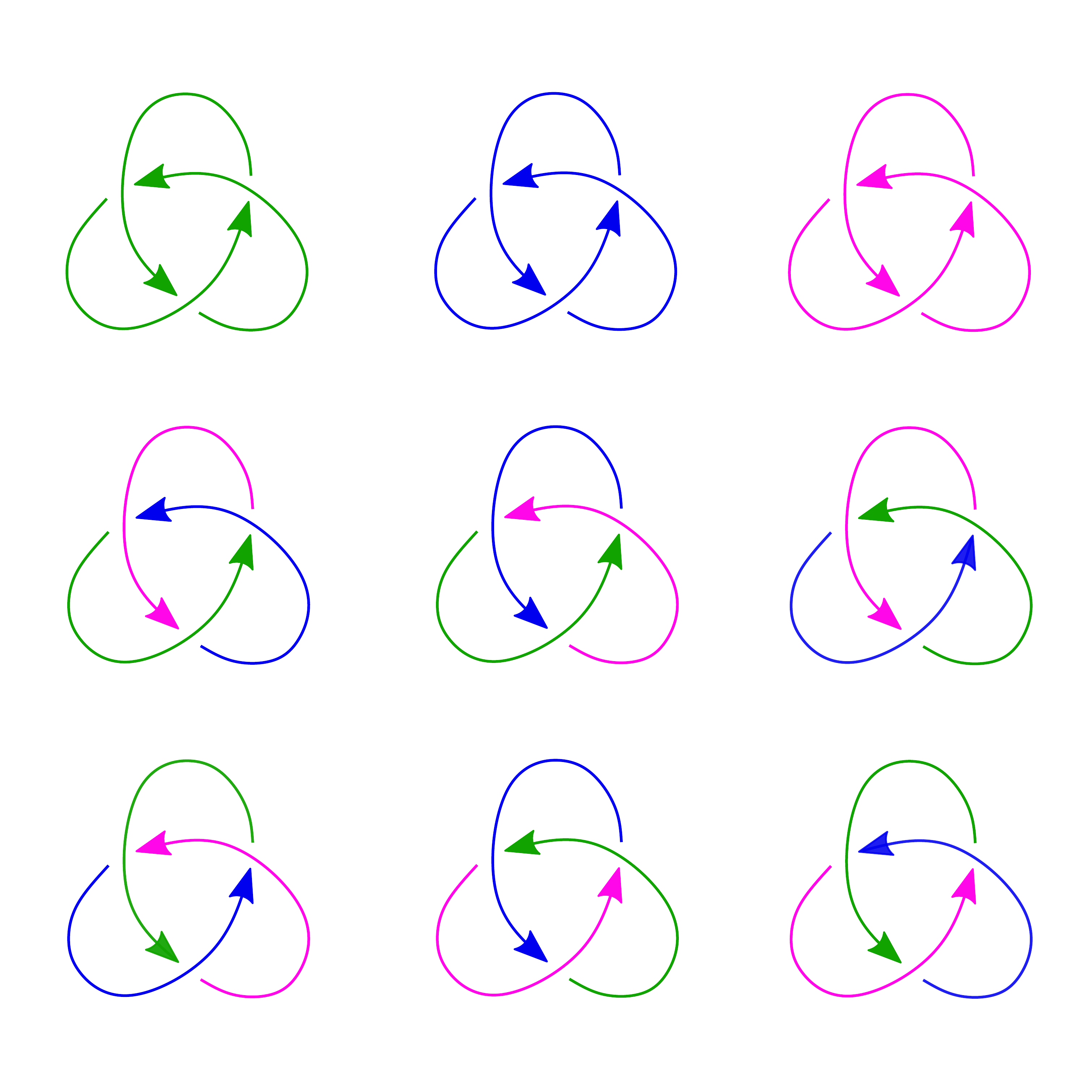}}\]

For example, the coloring in the center of the grid represents the homomorphism that maps $a \mapsto 0,$ $b \mapsto 1, c \mapsto 2$.

\end{example}

\section{\large\textbf{Quandle Coloring Quivers}}\label{QCQ}

In this section we define quivers (i.e., directed graphs possibly including 
loops and bigons)
associated to a pair $(L,X)$ consisting of an oriented link $L$ and finite
quandle $X$.

\begin{definition}
Let $X$ be a finite quandle and $L$ an oriented link. For any
set of quandle endomorphims $S\subset \mathrm{Hom}(X,X)$, the associated
\textit{quandle coloring quiver}, denoted $\mathcal{Q}_X^S(L)$,
is the directed graph with a vertex
for every element $f\in\mathrm{Hom}(Q(L),X)$ and an edge directed
from $f$ to $g$ when $g=\phi f$ for an element $\phi\in S$. Important
special cases include the case $S=\mathrm{Hom}(X,X)$, which we call the
\textit{full quandle coloring quiver} of $L$ with respect to $X$,
denoted $\mathcal{Q}_X(L)$,
and the case when $S=\{\phi\}$  is a singleton, which we
will denote by $\mathcal{Q}_X^{\phi}(L)$.
\end{definition}

The vertices in $\mathcal{Q}_X(L)$ can be identified with quandle homomorphisms
from the fundamental quandle of the link to the coloring quandle $X$ or,
equivalently, with $X$-colorings of a diagram of $L$. Then there is a directed
edge in $\mathcal{Q}_X(L)$ from a vertex $v_1$ to a vertex $v_2$ if there is
a quandle endomorphism $e:X\to X$ such that the diagram of quandle
homomorphisms
\[\includegraphics{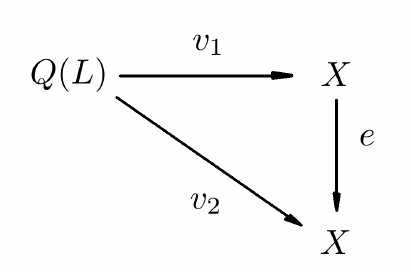}\]
commutes.

\begin{theorem}
Let $X$ be a finite quandle, $S\subset\mathrm{Hom}(X,X)$ and $L$ an oriented
link. Then the  quiver $\mathcal{Q}_X^{S}(L)$ is an invariant of $L$.
\end{theorem}

\begin{proof}
The quiver $\mathcal{Q}_X^{S}(L)$ is determined up to isomorphism by $X$ and $\mathrm{Hom}(Q(L),X)$ is an invariant of $L$.
\end{proof}

\begin{corollary}
Any invariant of directed graphs applied to $\mathcal{Q}_X^{S}(L)$ defines an
invariant of oriented links.
\end{corollary}

\begin{remark}
The full quandle counting quiver forms a category with vertices as objects and
directed paths modulo loops as morphisms; hence 
$\mathcal{Q}_X(L)$ is a categorification of the quandle counting invariant
$\mathrm{Hom}(Q(L),X)$. The decategorification map is the forgetful functor
from the quiver to set of vertices, $V(\mathcal{Q}_X(L))=\mathrm{Hom}(Q(L),X)$.
\end{remark}

\begin{example}\label{ex:1}
The Hopf link $L2a1$ has five colorings by the quandle $X$ specified by
the operation table as shown below.
\[\begin{array}{r|rrr}
\tr & 1 & 2 & 3 \\ \hline
  1 & 1 & 1 & 2 \\
  2 & 2 & 2 & 1 \\
  3 & 3 & 3 & 3
\end{array}\]
Then for example the quandle homomorphism $\phi_1:X\to X$ defined by
$\phi_1(1)=2,$ $\phi_1(2)=1,$ and $\phi_1(3)=3$ yields the following quandle
coloring quiver $\mathcal{Q}_X^{\phi_1}(L2a1)$
\[\scalebox{1}{\includegraphics{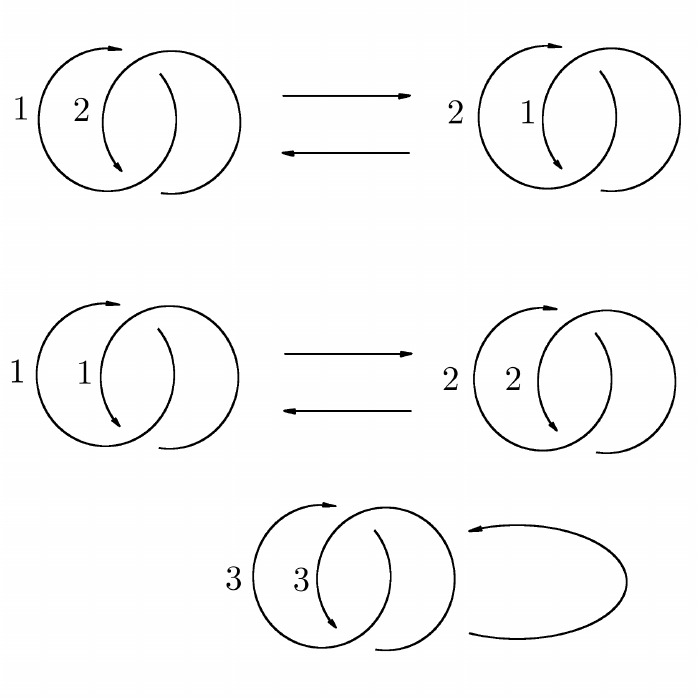}}
\quad \raisebox{1.25in}{or}\quad
\raisebox{0.5in}{\includegraphics{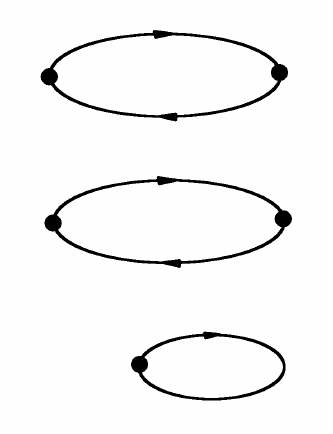}}\]
while the quandle homomorphism $\phi_2:X\to X$ defined by
$\phi_2(1)=\phi_2(2)=\phi_2(3)=3$ yields
$\mathcal{Q}_X^{\phi_2}(L2a1)$
\[\scalebox{1}{\includegraphics{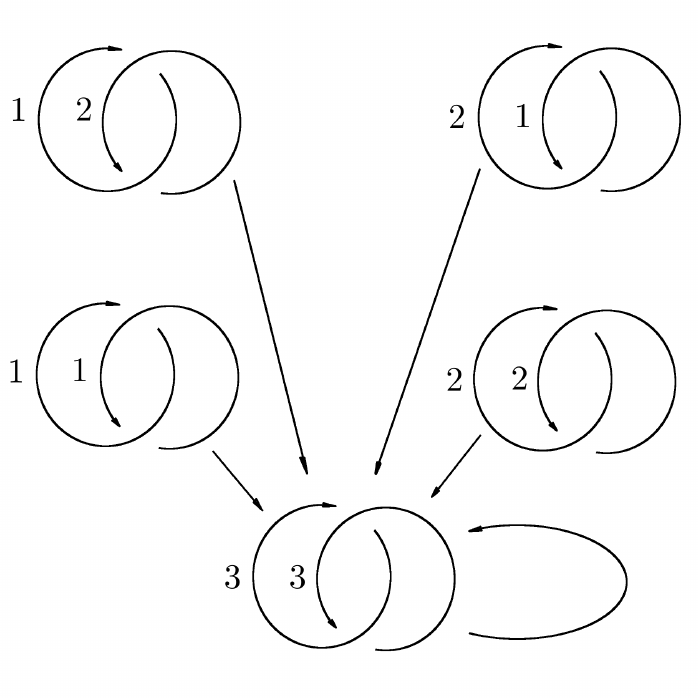}}
\quad \raisebox{1.25in}{or}\quad
\raisebox{0.5in}{\includegraphics{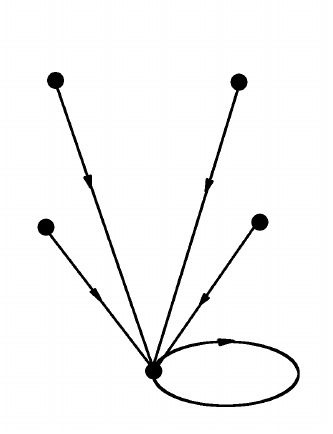}}.\]
Setting $S=\{\phi_3,\phi_4\}$ where $\phi_3(1)=\phi_3(2)=1$ and $\phi_3(3)=2$
and $\phi_4(1)=\phi_4(2)=2$ and $\phi_4(3)=1$ yields the quiver
$\mathcal{Q}_X^{S}(L2a1)$
\[\scalebox{1}{\includegraphics{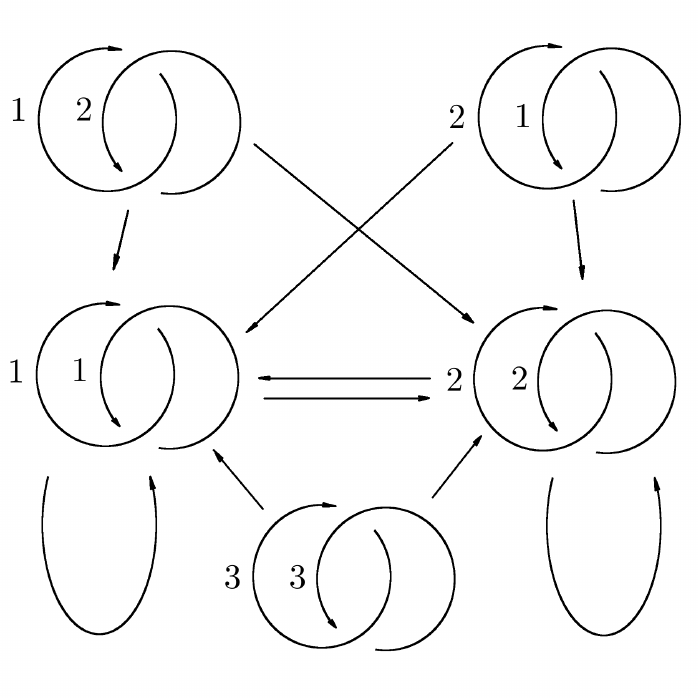}}
\quad \raisebox{1.25in}{or}\quad
\raisebox{0.5in}{\includegraphics{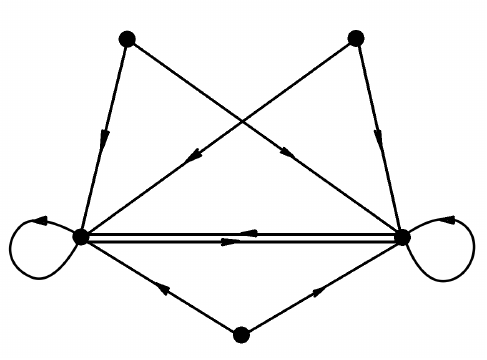}}.\]
\end{example}

\begin{remark}
Example \ref{ex:1} inspires a simple integer-valued invariant: the number of
components of $\mathcal{Q}_X^{S}(L)$. More generally, the ranks of
$H_0(\mathcal{Q}_X^{S}(L))$ and $H_1(\mathcal{Q}_X^{S}(L))$ are
integer-valued oriented link invariants.
\end{remark}

\begin{example}
Let $X$ be the quandle given by the operation table
\[\begin{array}{r|rrrr}
\tr & 1 & 2 & 3 & 4 \\ \hline
  1 & 1 & 3 & 1 & 3 \\
  2 & 4 & 2 & 4 & 2 \\
  3 & 3 & 1 & 3 & 1 \\
  4 & 2 & 4 & 2 & 4
\end{array}\] and let $\phi:X\to X$
be the endomorphism given by $\phi(1)=4$, $\phi(1)=2$, $\phi(1)=4$
and $\phi(1)=2$. Then the links $L6a1$ and $L6a5$
\[\includegraphics{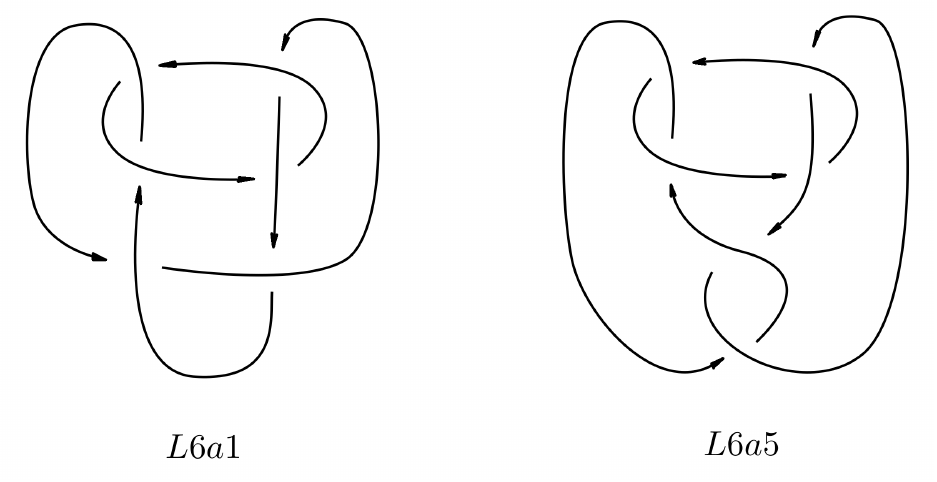}\]
 both have
$\Phi_X^{\mathbb{Z}}(L6a1)=\Phi_X^{\mathbb{Z}}(L6a5)=16$
$X$-colorings and hence are not distinguished by the quandle counting
invariant defined by $X$; however, the quandle coloring quivers
are non-isomorphic, distinguishing the links:
\[\begin{array}{cc}
\includegraphics{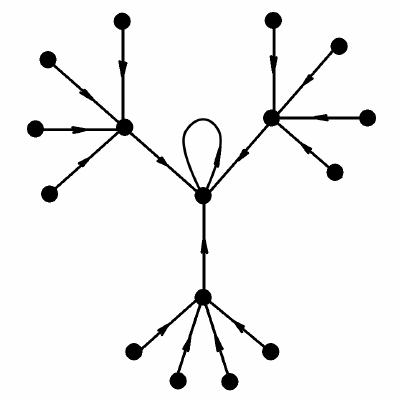} & \includegraphics{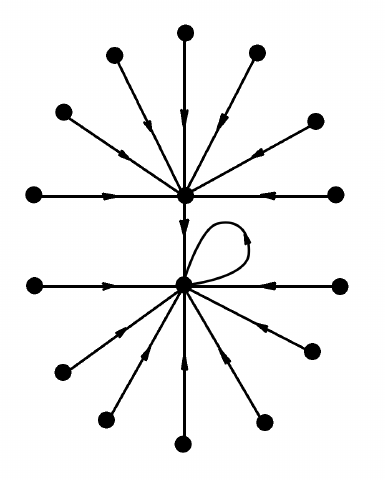} \\
\mathcal{Q}_X^{\phi}(L6a1) & \mathcal{Q}_X^{\phi}(L6a5)
\end{array}\]
In particular, this example demonstrates that $\mathcal{Q}_X^{\phi}(L)$
is a proper enhancement of the quandle counting invariant.
\end{example}

The full quandle coloring quiver has many multiple edges, so for simplicity
we can draw single directed edges labeled with multplicities.

\begin{example}\label{ex:2}
The full quandle counting quiver for the link and quandle in
Example \ref{ex:1} is
\[\includegraphics{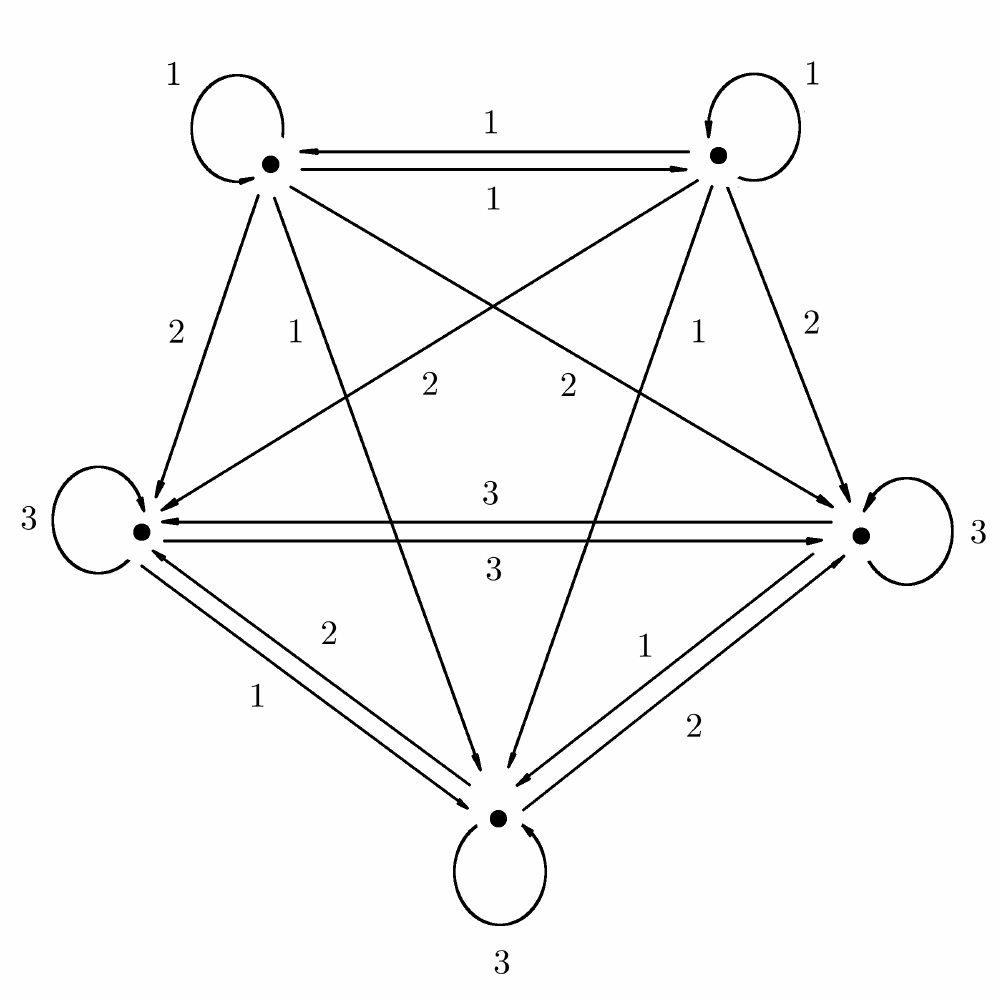}.\]
\end{example}

We observe that in a quandle coloring quiver, the out-degree
$\mathrm{deg}^-(f)$ of every vertex is the same, namely $|S|$;
however, different
vertices may have different in-degrees $\mathrm{deg}^+(f)$. This information
can be encoded as a polynomial knot invariant:

\begin{definition}
Let $X$ be a finite quandle, $S\subset\mathrm{Hom}(X,X)$ a set of quandle
endomorphisms, $L$ an oriented link and $\mathcal{Q}_X^S(L)$
the associated quandle coloring quiver with set of
vertices $V(\mathcal{Q}_X^S(L))$. Then the
\textit{in-degree quiver polynomial} of $L$ with respect to $X$ is
\[\Phi_X^{\mathrm{deg}^+,S}(L)
=\sum_{f\in V(\mathcal{Q}_X^{S}(L))} u^{\mathrm{deg}^+(f)}.\]
If $S=\{\phi\}$ is a singleton we will write $\Phi_X^{\mathrm{deg}^+,S}(L)$
as $\Phi_X^{\mathrm{deg}^+,\phi}(L)$ and if $S=\mathrm{Hom}(X,X)$ we will write
$\Phi_X^{\mathrm{deg}^+,S}(L)$ as $\Phi_X^{\mathrm{deg}^+}(L)$.
\end{definition}

\begin{example}
In Example \ref{ex:1}, the Hopf link has in-degree quiver polynomial
$\Phi_X^{\mathrm{deg}^+}(L2a1)=2u^2+u^7+2u^{12}$ with respect to the quandle
$X$, and in-degree quiver polynomials
$\Phi_X^{\mathrm{deg}^+,\phi_1}(L2a1)=5u$,
$\Phi_X^{\mathrm{deg}^+,\phi_2}(L2a1)=4+u^5$
and
$\Phi_X^{\mathrm{deg}^+,S}(L2a1)=2u^5+3$
with respect to the quandle homomorphisms $\phi_1$ and $\phi_2$ and
subset $S=\{\phi_3,\phi_4\}$.
\end{example}

The following example shows that in-degree quiver polynomials are proper
enhancements of the quandle counting invariant.

\begin{example}
Let $X$ be the quandle with operation table below. The reader can verify that
the map $\phi:X\to X$ defined by setting $\phi(1)=\phi(2)=\phi(3)=6$,
$\phi(4)=4$ and $\phi(5)=\phi(6)=5$ is a quandle endomorphism.
\[\begin{array}{r|rrrrrr}
\tr & 1 & 2 & 3 & 4 & 5 & 6 \\ \hline
1 & 1 & 3 & 2 & 1 & 3 & 2 \\
2 & 3 & 2 & 1 & 2 & 1 & 3 \\
3 & 2 & 1 & 3 & 3 & 2 & 1 \\
4 & 4 & 4 & 4 & 4 & 4 & 4 \\
5 & 6 & 6 & 6 & 5 & 5 & 5 \\
6 & 5 & 5 & 5 & 6 & 6 & 6
\end{array}\]
Then the links $L7a3$ and $L7a4$
\[\includegraphics{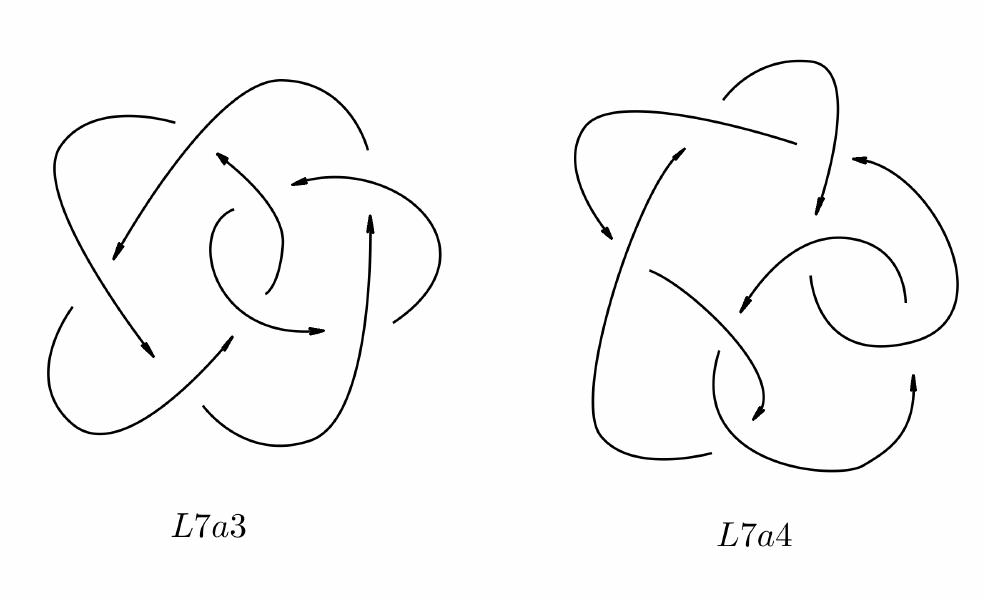}\]
 have equal quandle counting invariant
values $\Phi_X^{\mathbb{Z}}(L7a3)=\Phi_X^{\mathbb{Z}}(L7a4)=30$
but are distinguished by the in-degree quandle quiver polynomials
\begin{eqnarray*}
\Phi_X^{\mathrm{deg^+,\phi}}(L7a3) & = &
22+u+2u^2+2u^3+u^4+u^6+u^9
 \ne \\
\Phi_X^{\mathrm{deg}^+,\phi}(L7a4) & = &
21+u+2u^2+3u^3+u^4+2u^6.
\end{eqnarray*}
\end{example}

\begin{example}
Let $X$ be the quandle with operation table
\[\begin{array}{r|rrrrrrrr}
\tr & 1 & 2 & 3 & 4& 5 & 6 & 7 & 8 \\ \hline
1 & 1 &4 &2 &3 &3 &2 &1 &4 \\
2 & 3 &2 &4 &1 &4 &1 &2 &3 \\
3 & 4 &1 &3 &2 &1 &4 &3 &2 \\
4 & 2 &3 &1 &4 &2 &3 &4 &1 \\
5 & 8 &8 &8 &8 &5 &5 &5 &5 \\
6 & 5 &5 &5 &5 &6 &6 &6 &6 \\
7 & 7 &7 &7 &7 &7 &7 &7 &7 \\
8 & 6 &6 &6 &6 &8 &8 &8 &8
\end{array}\]
and let $\phi:X\to X$ be given by
$\phi(1)=\phi(2)=\phi(3)=\phi(4)=7$, $\phi(5)=\phi(6)=\phi(8)=5$
and $\phi(7)=6$. Then we compute via our \texttt{python} code the
following $\Phi_X^{\mathrm{deg}^+,\phi}$ values on the prime links of up to
seven crossings from \cite{KA}.
\[\begin{array}{r|l}
L & \Phi_X^{\mathrm{deg}^+,\phi}(L) \\ \hline
L2a1 & u^9 + 3u^4 + 2u^3 + u + 21 \\
L4a1 & u^{16} + u^9 + 2u^4 + 2u^3 + u + 33 \\
L5a1 & 2u^{12} + u^9 + 3u^4 + 2u^3 + u + 43 \\
L6a1 & u^9 + 3u^4 + 2u^3 + u + 21 \\
L6a2 & 2u^{12} + u^9 + 3u^4 + 2u^3 + u + 43 \\
L6a3 & u^{16} + 2u^{12} + u^9 + 2u^4 + 2u^3 + u + 55 \\
L6a4 & u^{27} + 3u^{12} + 3u^9 + 7u^4 + 3u^3 + u + 110 \\
L6a5 & u^{27} + 3u^9 + 7u^4 + 3u^3 + u + 77 \\
L6n1 & u^{27} + u^{16} + 3u^9 + 6u^4 + 3u^3 + u + 89 \\
L7a1 & u^{16} + 2u^{12} + u^9 + 2u^4 + 2u^3 + u + 55 \\
L7a2 & u^{16} + u^9 + 2u^4 + 2u^3 + u + 33 \\
L7a3 & u^{16} + u^{12} + u^9 + 2u^4 + 2u^3 + u + 44 \\
L7a4 & u^{16} + 2u^{12} + u^9 + 2u^4 + 2u^3 + u + 55 \\
L7a5 & u^9 + 3u^4 + 2u^3 + u + 21 \\
L7a6 & u^9 + 3u^4 + 2u^3 + u + 21 \\
L7a7 & u^{27} + u^{16} + 2u^{12} + 3u^9 + 6u^4 + 3u^3 + u + 111 \\
L7n1 & u^{16} + u^9 + 2u^4 + 2u^3 + u + 33 \\
L7n2 & u^{16} + u^{12} + u^9 + 2u^4 + 2u^3 + u + 44.
\end{array}\]

\end{example}

The in-degree quiver polynomial is good not just for multi-component links
but can distinguish prime knots as well.

\begin{example}
Let $X$ be the quandle given by the operation table below.
\[\begin{array}{r|rrrrrrrrr}
\tr & 1 & 2 & 3 & 4 & 5 & 6& 7 & 8 & 9 \\ \hline
1 & 1 &3 &2 &7 &9 &8 &4 &6 &5 \\
2 & 3 &2 &1 &9 &8 &7 &6 &5 &4 \\
3 & 2 &1 &3 &8 &7 &9 &5 &4 &6 \\
4 & 8 &7 &9 &4 &6 &5 &2 &1 &3 \\
5 & 7 &9 &8 &6 &5 &4 &1 &3 &2 \\
6 & 9 &8 &7 &5 &4 &6 &3 &2 &1 \\
7 & 6 &5 &4 &3 &2 &1 &7 &9 &8 \\
8 & 5 &4 &6 &2 &1 &3 &9 &8 &7 \\
9 & 4 &6 &5 &1 &3 &2 &8 &7 &9
\end{array}\]
The prime knots $8_{10}$ and $8_{18}$
\[\includegraphics{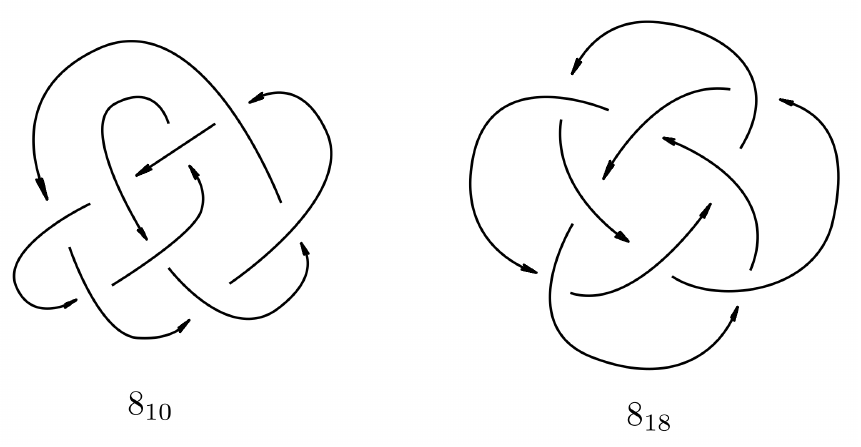}\]
are then distinguished according to our \texttt{python} computations
by the in-degree full quiver polynomials
\[
\Phi_X^{\mathrm{deg}^+}(8_{10})=9u^{189} + 18u^{108} + 54u^{54}
\ne
\Phi_X^{\mathrm{deg}^+}(8_{18})=9u^{297} + 72u^{54}
\]
as well as several of the in-degree quiver polynomials
associated to various endomorphisms, e.g.
\[
\Phi_X^{\mathrm{deg}^+,\phi}(8_{10})=9u^{9} + 72
\ne
\Phi_X^{\mathrm{deg}^+,\phi}(8_{18})=3u^{27} + 78
\]
for the quandle endomorphism defined by
$\phi(1)=\phi(2)=\phi(3)=9$, $\phi(4)=\phi(5)=\phi(6)=7$ and
$\phi(7)=\phi(8)=\phi(9)=8$.
\end{example}

\section{\large\textbf{Questions}}\label{Q}

We end with some questions for future research.

Many more enhancements of $\Phi_X^{\mathbb{Z}}$ can be defined from
$\mathcal{Q}_X$; we have considered only some of the most basic. Other
examples include the eigenvalues and characteristic polynomial of the
adjacency matrix of $\mathcal{Q}_X$, topological invariants such as
number of components or rank of first homology, the Tutte polynomial of
the graph and much more. How are these enhancements related to other invariants?

What enhancements can be derived functorially from the categorical structure of
$\mathcal{Q}_X$, e.g. via representations into vector space or module
categories?

If $Y\subset X$ is a subquandle, what is the precise relationship between 
the invariants derivable from $\mathcal{Q}_Y(L)$ and those derivable from
 $\mathcal{Q}_X(L)$?

\bibliography{kc-sn}{}
\bibliographystyle{abbrv}

\bigskip

\noindent
\textsc{Department of Mathematics \\
Harvey Mudd College\\
301 Platt Boulevard \\
Claremont, CA 91711
}

\bigskip

\noindent
\textsc{Department of Mathematical Sciences \\
Claremont McKenna College \\
850 Columbia Ave. \\
Claremont, CA 91711}

\end{document}